\documentclass[12pt, reqno]{amsart}
\usepackage{lmodern}
\usepackage{amsmath,amsthm,amscd,amsfonts,amssymb,graphicx,color}
\usepackage[bookmarksnumbered,colorlinks,plainpages]{hyperref}
\usepackage{enumitem}
\usepackage{bm}
\usepackage{amsmath,amssymb}
\usepackage{graphicx}
\usepackage{amsmath, amssymb}
\usepackage{tikz}
\usepackage{cases}
\usetikzlibrary{positioning, arrows.meta}
\usepackage[a4paper,right=1in,left=1in,top=1.5in,bottom=1.5in]{geometry}
\usepackage{ragged2e}
\hypersetup{colorlinks=true,linkcolor=black, anchorcolor=black,
	citecolor=black, urlcolor=black, filecolor=magenta, pdftoolbar=true}
    \newtheorem{theorem}{Theorem}[section]

\newtheorem{proposition}{Proposition}[section]

\theoremstyle{definition}
\newtheorem{definition}{Definition}[section]
\newtheorem{example}{Example}[section]

\theoremstyle{remark}
\newtheorem{remark}{Remark}[section]
\numberwithin{equation}{section}

\begin{document}
\setcounter{page}{1}

\noindent 
\begin{center}
   {\textbf{\Large On the Characterization of gH-partial derivatives and gH-Product for Interval-Valued Functions}}\\[0.3cm]
   {\bf Amir Suhail\textsuperscript{a}, Tauheed\textsuperscript{a} and Akhlad Iqbal\textsuperscript{a}\textsuperscript{*} }
\end{center}
\noindent
\textit{\textsuperscript{a}Department of Mathematics, Aligarh Muslim University, Aligarh-202002, Uttar Pradesh, India}
\let\thefootnote\relax\footnotetext{\textsuperscript{*}Corresponding author: Akhlad Iqbal.\\
Email addresses:amir.mathamu@gmail.com (Amir Suhail),\\
tauheedalicktd@gmail.com (Tauheed),\\
akhlad.mm@amu.ac.in (Akhlad Iqbal).}

\bigskip

\textbf{Abstract}: In this paper, we show by a counterexample that the \emph{gH-partial derivative} of interval-valued functions (IVFs) may exist even when the partial derivative of the end point functions do not. Next, we introduce the \emph{gH-partial derivative}  in terms of \emph{gH-derivative} and discuss its complete characterization. Furthermore, we introduce the \emph{gH-product} of a vector with an n-tuples of intervals and illustrate by a suitable example that our definition refines the definition existing in the literature. To illustrate and validate these definitions, we provide several non-trivial examples.\vspace{0.2cm}   \vspace{0.2 cm}

\section{\textbf{INTRODUCTION}}
Uncertainty and imprecision are inherent in many real-world systems, particularly in engineering, economics, and decision sciences. In recent decades, the study of interval-valued functions (IVFs) has gained significant attention due to their applicability in modeling uncertainty and imprecision in various scientific and engineering problems \cite{Doolittle,Izhar}. Traditional real-valued analysis often falls short in contexts where input data, parameters, or outputs are not precisely known, leading to the need for interval-based mathematical frameworks. The foundation for such an approach is the theory of interval arithmetic, where numbers are represented as intervals instead of exact values, allowing the accommodation of ambiguity in computation. \vspace{.2cm}\\
 A major challenge in working with IVFs is that classical operations like subtraction and differentiation do not naturally extend to the interval setting. For example, standard interval subtraction can yield non-degenerate intervals even when subtracting an interval from itself \cite{Moore1966,Moore1979}. To resolve this, the \emph{Hukuhara difference} \cite{Hukuhara} was introduced, but it only applies under restrictive conditions. Later, \emph{Markov}~\cite{Markov} introduced difference between two intervals in a different way,  and discusssed calculus for IVFs of a real variable. Afterwards, \emph{Stefanini and Bede}~\cite{Stefanini2009} proposed the \emph{generalized Hukuhara (gH) difference}, which provides the same framework for interval subtraction as \emph{Markov} \cite{Markov}. Based on this foundation, they also introduced the notion of \emph{gH-differentiability}, enabling the development of calculus for IVFs of a real variable. The notion of \emph{gH-differentiability} for IVFs has also been extended to \emph{fuzzy-valued functions}, see (\cite{StefaniniBede2013}, \cite{Chalco2016}, \cite{Stefanini2019}). \vspace{.2cm} \\
Several researchers (\cite{StefaniniBede2013}, \cite{Chalco2016}, \cite{Stefanini2009}, \cite{Stefanini2019}) have characterized the \emph{$gH$-differentiability} for IVFs by using \emph{endpoint differentiabilty}. However,  \emph{Dong Qiu}~\cite{D. Qiu} presented a complete characterization of the \emph{gH-differentiability} and shown that the \emph{gH-differentiability of IVFs is not equivalent to the end-point differentiability}. \textit{Stefanini and Arana} ~\cite{Stefanini2019} extended the $gH$-differentiability for several variables to encompass the total and directional $gH$-differentiability, including the partial gH-differentiability. Partial differentiability and gradient for IVF have been defined by Ghosh et al. \cite{Ghosh,Ghosh2019}. Later, \textit{Osuna}~\cite{Gomez2022,Costa2022} presented a new definition of \emph{gH-differentiability} for IVFs of several variables by introducing a \emph{quasilinear interval approximation}. More recently, \emph{Bhat et al.}~\cite{Bhat2024, Bhat2025a, Bhat2025b} extended these ideas to Riemannian manifolds, establishing optimality conditions and derivative structures for IVFs in geometrically complex spaces.\vspace{.2cm}\\
In recent decades, the theory of calculus for fuzzy functions and ordinary differential equations with fuzzy parameters has been extensively investigated, both from theoretical and numerical perspectives \cite{Alikhani2014,Alikhani2017}. 
Building on this, the study of applied problems involving uncertain data has motivated the formulation of fuzzy partial differential equations. Nevertheless, compared to the case of single-variable functions \cite{Stefanini2009}, relatively less progress has been made in the analysis of multivariable fuzzy functions and the corresponding partial differential equations with fuzzy data \cite{Allahviranloo2006,Allahviranloo2015}.\vspace{0.4cm}\\
As we know that the \emph{$gH$-partial derivatives} for IVFs and \emph{fuzzy-valued functions} are related to each other, see \cite{Stefanini2019}. However, the earlier literature does not give the complete characterization of \emph{$gH$-partial derivatives} for IVFs as well as for \emph{fuzzy-valued functions}. In this paper we have demonstrated the \emph{complete characterization of $gH$-partial derivatives  in terms of gH-derivatives for the IVFs}, which can be extended for the \emph{fuzzy-valued functions} as well. Furthermore, we have defined a product of a vector with n-tuples of intervals named as  \emph{gH-product}. Theoretical constructs are validated through a series of non-trivial illustrative examples. This work aims to advance the mathematical foundations of interval as well as fuzzy analysis and provide new tools for the analysis and optimization of uncertain systems.\vspace{.2cm}\\
\noindent
The paper is structured as follows: Section~2 outlines the required preliminaries and fundamental definitions. Section~3 is divided into two parts: the first addresses the \emph{gH-partial derivative}, 
while the second focuses on the \emph{gH-product} and its properties. 
Illustrative examples are included to support the theoretical findings. 
Concluding remarks and prospects for future work are provided at the end.
\\

  \section{\textbf{PRELIMINARIES}} 
Let \( \mathcal{I}(\mathbb{R}) \) represent the collection of all closed and bounded intervals in \( \mathbb{R} \).
For any interval \( \mathcal{K} \in \mathcal{I}(\mathbb{R}) \), it is defined as:

\[
\mathcal{K} = [k^L, k^U] \quad \text{where} \quad k^L, k^U \in \mathbb{R} \quad \text{and} \quad k^L \leq k^U.
\]

Given two intervals \( \mathcal{K}_1 = [k_1^L, k_1^U] \) and \( \mathcal{K}_2 = [k_2^L,k_2^U] \) in \( \mathcal{I}(\mathbb{R}) \), their sum is defined as:

\[
\mathcal{K}_1 + \mathcal{K}_2 = [k_1^L +k_2^L, k_1^U +k_2^U].
\]

The negation of \( \mathcal{K}_1 \) is given by:

\[
-\mathcal{K}_1 = [-k_1^U, -k_1^L].
\]

Consequently, the difference between \( \mathcal{K}_1 \) and \( \mathcal{K}_2 \) is expressed as:

\[
\mathcal{K}_1 - \mathcal{K}_2 = \mathcal{K}_1 + (-\mathcal{K}_2) = [k_1^L -k_2^U, k_1^U -k_2^L].
\]

Additionally, scalar multiplication of \( \mathcal{K} \) by a real number \( \mathit{p} \) is defined as:

\[
\mathit{p}\mathcal{K} = 
\begin{cases} 
[\mathit{p}k^L, \mathit{p}k^U] & \text{if } \mathit{p} \geq 0, \\
[\mathit{p}k^U, \mathit{p}k^L] & \text{if } \mathit{p} < 0.
\end{cases}
\]
\noindent
This summarizes the fundamental operations on intervals within the set \( \mathcal{I}(\mathbb{R}) \).\vspace{-0.2cm}\\

For a deeper understanding of interval analysis, the interested reader is encouraged to consult the foundational works by Moore~\cite{Moore1966,Moore1979}, as well as the comprehensive treatment provided by Alefeld and Herzberger~\cite{AlefeldHerzberger}.

The Hausdorff distance between two intervals \( \mathcal{K}_1= [k_1^L, k_1^U] \) and \(\mathcal{K}_2 = [k_2^L,k_2^U] \) is defined as:

\[
d_H(\mathcal{K}_1, \mathcal{K}_2) = \max\{|k_1^L -k_2^L|, |k_1^U -k_2^U|\}.
\]

A limitation of  standard interval subtraction is that, for any interval \(\mathcal{K} \in \mathcal{ I(\mathbb{R})} \), the result of \( \mathcal{K -K}\) is not equal to zero. For instance, if \( \mathcal{K}= [0, 1] \), then:

\[
\mathcal{K - K} = [0, 1] - [0, 1] = [-1, 1] \neq 0.
\]

To resolve this issue, the \textit{Hukuhara difference} between two intervals \( \mathcal{K}_1 = [k_1^L, k_1^U] \) and \(\mathcal{K}_2= [k_2^L,k_2^U] \) is introduced as:

$$\mathcal{K}_1 \ominus \mathcal{K}_2 = [k_1^L -k_2^L, k_1^U -k_2^U].$$

With this definition, for any interval \( \mathcal{K} \in \mathcal{ I(\mathbb{R})} \), \(\mathcal{K \ominus K}= 0 \). However, the Hukuhara difference is not always valid for arbitrary intervals. For example, \( [0, 4] \ominus [0, 10] = [0, -6] \), which is not an interval since the lower bound exceeds the upper bound. This highlights a restriction in the applicability of the Hukuhara difference.

To overcome this limitation, Stefanini et al. \cite{Stefanini2009} proposed the \textit{generalized Hukuhara difference} (gH-difference) for \( \mathcal{K}_1 \) and \(\mathcal{K}_2 \), which is defined as:

\[
\mathcal{K}_1 \ominus_{gH} \mathcal{K}_2= \mathcal{K}_3 \iff 
\begin{cases} 
    (i) & \mathcal{K}_1 =\mathcal{K}_2 + \mathcal{K}_3, \quad \text{or} \\ 
    (ii) & \mathcal{K}_2 = \mathcal{K}_1 - \mathcal{K}_3. 
\end{cases}
\]

In case \((i)\), the gH-difference is equivalent to the Hukuhara difference (H-difference).

 For any two intervals \( \mathcal{K}_1 = [k_1^L, k_1^U] \) and \(\mathcal{K}_2 = [k_2^L,k_2^U] \), the gH-difference \( \mathcal{K}_1 \ominus_{gH}\mathcal{K}_2 \) always exists and is uniquely determined. Moreover, the following properties hold:

\begin{equation*}
    \begin{aligned}
        \mathcal{K}_1 \ominus_{gH} \mathcal{K}_1 = [0, 0]
    \end{aligned}
\end{equation*}  
\begin{equation*}
    \begin{aligned}
        \mathcal{K}_1 \ominus_{gH} \mathcal{K}_2 = \left[\min\{k_1^L -k_2^L, k_1^U -k_2^U\}, \max\{k_1^L -k_2^L, k_1^U -k_2^U\}\right].
    \end{aligned}
\end{equation*}

This generalized approach ensures that the difference between intervals is always well-defined and resolves the issues associated with the standard Hukuhara difference.

\begin{definition} \cite{Wu2007}.A map \( \tilde{h}: \mathbb{R}^n \to \mathcal{I}(\mathbb{R}) \) is an IVF if, for each  \( \mathit{x}  \in \mathbb{R}^n \),

\[
\tilde{h}(\mathit{x} ) = \left[\tilde{h}^L(\mathit{x} ), \tilde{h}^U(\mathit{x} )\right],
\]

where \( \tilde{h}^L, \tilde{h}^U : \mathbb{R}^n \rightarrow \mathbb{R} \) are real-valued functions such that \( \tilde{h}^L(\mathit{x} ) \leq \tilde{h}^U(\mathit{x} ) \), $\forall$ \( \mathit{x}  \in \mathbb{R}^n \).
\end{definition} 

 Building on this formal structure, Wu \cite{Wu2007} introduced a rigorous extension of classical calculus into the interval domain: the notions of continuity, limit, and two distinct forms of differentiability for interval-valued mappings. Next, we give a definition and a result that will be used in building our main results.

\vspace{0.2cm}

\begin{definition}\cite{Stefanini2009}
Let $\mathit{x} _0 \in (a, b)$.Then  IVF $\tilde{h} : (a, b) \to \mathcal{I}(\mathbb{R})$ is said to be gH-differentiable at $\mathit{x} _0$  if 

\begin{equation*}
    \tilde{h}'(\mathit{x} _0) = \lim_{t \to 0} \frac{\tilde{h}(\mathit{x} _0 + t) \ominus_{gH}\tilde{h}(\mathit{x} _0)}{t}
\end{equation*}
exists and $\tilde{h}'(\mathit{x} _0)$
is called gH-derivative of $\tilde{h}~at~ \mathit{x}_0$
\end{definition}

\begin{proposition}\label{prop2.1}{\cite{Stefanini2009}}
Given $\mathcal{K}_1, \mathcal{K}_2 \in \mathcal{I}(\mathbb{R})$ and $\nu \in \mathbb{R}$, it follows that

     $\nu \cdot \left(\mathcal{K}_1 \ominus_{gH} \mathcal{K}_2\right) 
    = \nu  \cdot \mathcal{K}_1 \,\ominus_{gH}\, \nu \cdot \mathcal{K}_2$

\end{proposition}

\section{ \textbf{MAIN RESULTS }}
 \noindent
In this section, we define the \emph gH-partial derivatives of an IVF by using  the  definitions of gH-differentiability of an IVF. Furthermore,we introduce the $gH$-product of a vector with n -tuples of intervals, providing a compatibility of a vector with n-tuples of intervals. These foundational definitions are essential for the theoretical development and subsequent analysis for future work.
\begin{center}
   \item \subsection{\textit{gH- partial derivatives and gH-Gradient for IVFs}}
\end{center}
 The notion of the gH-gradient for IVFs using partial derivatives is defined as follows: \\
  Let  $\tilde{h} : \mathcal{S} \overset{\text{open}}{\subseteq} \mathbb{R}^n \rightarrow \mathcal{I}(\mathbb{R})$.
  Let $\mathit{x}  = (\mathit{x} _1, \mathit{x} _2, \dots, \mathit{x} _n)$ $\in \mathbb{R}^n$ and $t\in \mathbb{R}$. Then we have,
  \begin{equation*} 
  	\frac{\partial \tilde{h}}{\partial \mathit{x}_i} = \lim_{t \to 0} \frac{\tilde{h}(\mathit{x}_1, \mathit{x}_2, \dots, \mathit{x}_i + t, \dots, \mathit{x}_n) \ominus_{gH} \tilde{h}(\mathit{x}_1, \mathit{x}_2, \dots, \mathit{x}_i, \dots, \mathit{x}_n)}{t}.
  \end{equation*}\\
    
  \noindent
   The gH-gradient of $\tilde{h}(\mathit{x} )$ is represented as below.
  \begin{equation*}
  	\nabla_{gH} \tilde{h}(\mathit{x} ) = \left(\frac{\partial \tilde{h}}{\partial \mathit{x}_1}, \frac{\partial \tilde{h}}{\partial \mathit{x}_2} , \ldots, \frac{\partial \tilde{h}}{\partial \mathit{x}_n} \right)^T.
  \end{equation*}
  
\noindent
Stefanini et al. \cite{Stefanini2019} and Ghosh et al. \cite{Ghosh, Ghosh2019} defined the $gH$-partial derivative of $\tilde{h}$ as follows:

 \vspace{0.2cm}
\noindent

\begin{definition}\cite{Ghosh}
For the function $\Phi:\mathbb{R}\rightarrow\mathcal{I}(\mathbb{R})$ defined as:
\begin{equation*}
    \Phi(\mathit{\zeta} ) =\tilde{h}(\mathit{x} _1, \mathit{x} _2, \dots, \mathit{x}_{i-1}, \mathit{\zeta}, \mathit{x}_{i+1}, \dots, \mathit{x} _n).
\end{equation*}

If $\Phi$ is $gH$-differentiable at $\mathit{x} _i $, that is,
\begin{equation*}
    \lim_{t \to 0} \frac{\Phi(\mathit{x} _i + t) \ominus_{gH} \Phi(\mathit{x} _i)}{t}=\frac{d\Phi}{d\mathit{\zeta}}|_{\mathit{\zeta}=\mathit{x}_i}=\Phi^{'}(\mathit{x}_i) ,
\end{equation*}
exists, then $\tilde{h} $ is said to have the gH-partial derivative \textit{w.r.t.} $\mathit{x} _i$ and can be obtained by: 
 \begin{align*}          
	\frac{\partial \tilde{h}}{\partial\mathit{x}_i}= \big[  \min\{\frac{\partial \tilde{h}^L}{\partial\mathit{x}_i}, \frac{\partial \tilde{h}^U}{\partial\mathit{x}_i} \},\max\{\frac{\partial \tilde{h}^L}{\partial\mathit{x}_i}, \frac{\partial \tilde{h}^U}{\partial\mathit{x}_i} \} \big].
\end{align*}
\end{definition}
 
 \vspace{0.2cm} 
 \noindent
 It is to be noted, this definition requires the existence of partial derivatives of the endpoint functions.
 However, the partial derivatives of $\tilde{h}$ may exist even when the partial derivatives of the endpoint functions do not (See Ex. 3.1).
 For that, motivated by D. Qiu~\cite{D. Qiu}, we introduce the gH-partial derivative in terms of gH-derivative of $\Phi$. 
 
 \vspace{0.2cm} 
 \noindent
  First, we define the right gH-partial derivative and left gH-partial derivative of $\tilde{h}$ as follows:
\begin{equation*}
    \begin{aligned}
        \frac{\partial \tilde{h}_+}{\partial \mathit{x}_i} &= \textit{right gH-partial derivative of }~ \tilde{h} \textit{ w.r.t. } \mathit{x}_i\\
        &=\textit{right gH-derivative of} ~ \Phi \textit{ at }~\mathit{\zeta} = \mathit{x}_i
    \end{aligned}
\end{equation*}

\begin{equation*}
    \begin{aligned}
        \frac{\partial \tilde{h}_-}{\partial \mathit{x}_i} &= \textit{left gH-partial derivative of }~ \tilde{h} \textit{ w.r.t. } \mathit{x}_i\\
        &=\textit{left gH-derivative of} ~ \Phi \textit{ at }~\mathit{\zeta} = \mathit{x}_i
    \end{aligned}
\end{equation*}
\\
\noindent
    Let $f : \mathrm{X} \overset{\text{open}}{\subseteq} \mathbb{R} \to \mathbb{R}$ be a real-valued function and $\mathit{x}\in \mathrm{X}$. Define the function $\gamma_f : \mathbb{R} \setminus \{0\} \to \mathbb{R}$ by
\[
\gamma_f(t) = \frac{f(\mathit{x} + t) - f(\mathit{x})}{t},
\] 
where $t$ satisfies $\mathit{x} + t \in \mathrm{X}$.

\begin{definition}\cite{D. Qiu}
Consider two real-valued functions $g_{1}$ and $g_{2}$ defined on $(x_{0},\,x_{0}+\delta)$. 
We say that $g_{1}$ and $g_{2}$ are \emph{right complementary} at $x_{0}$ if the following conditions hold:
\begin{enumerate}[label=(\roman*)] 
   \item The set of cluster points of $g_1$ and $g_2$ on the right of $x_0$ i.e. $C_{R(x_{0})}(g_{1})$ and  $C_{R(x_{0})}(g_{2}),$ satisfy,
    \[C_{R(x_{0})}(g_{1}) = C_{R(x_{0})}(g_{2}) = \{\mathit{k}^L, \mathit{k}^U\},\]
    where $\mathit{k}^L,\mathit{k}^U \in \mathbb{R}$ with $\mathit{k}^L < \mathit{k}^U$.\\
    \item $\lim_{t \to 0^{+}} \min \{ g_{1}(x_{0}+t),\, g_{2}(x_{0}+t) \} = \mathit{k}^L, and\\
           \lim_{t \to 0^{+}} \max \{ g_{1}(x_{0}+t),\, g_{2}(x_{0}+t) \} = \mathit{k}^U.$
    \end{enumerate}
Analogously, we can define  left complementary at $x_0$.
\end{definition}

\vspace{0.2cm} 
\noindent
Now, we demonstrate the complete characterization of the $gH$-partial derivative of IVF $\tilde{h}$ in terms of their endpoint functions.

\begin{theorem} \label{theorem 3.1 p3}
  Let  $\tilde{h} : \mathcal{S} \overset{\text{open}}{\subseteq} \mathbb{R}^n \rightarrow \mathcal{I}(\mathbb{R})$. Also, suppose that $\Phi : \mathbb{R} \to \mathcal{I}(\mathbb{R})$ defined as \begin{equation*}
    \Phi(\mathit{\zeta} ) =\tilde{h}(\mathit{x} _1, \mathit{x} _2, \dots, \mathit{x}_{i-1}, \mathit{\zeta} , \mathit{x}_{i+1}, \dots, \mathit{x} _n).
\end{equation*}
The $gH$-partial derivative of $\tilde{h}$ with respect to $\mathit{x} _i$ exist iff one of the following cases holds:   
\end{theorem}  
\noindent
\begin{itemize}
    \item[(i)] $(\Phi^{L}( \mathit{x}_i))_+'=\frac{\partial \tilde{h}^L_+}{\partial \mathit{x}_i} $, $(\Phi^{U}( \mathit{x}_i))_+'=\frac{\partial \tilde{h}^U_+}{\partial \mathit{x}_i}$, $(\Phi^{L}( \mathit{x}_i))_-'=\frac{\partial \tilde{h}^L_-}{\partial \mathit{x}_i}$ and $(\Phi^{U}( \mathit{x}_i))_-'=\frac{\partial \tilde{h}^U_-}{\partial \mathit{x}_i}$ exist, and
    \begin{equation*}
        \begin{aligned}
            \frac{\partial \tilde{h}}{\partial\mathit{x}_i} =\Phi'(\mathit{x}_i) &= \big[ \min\{(\Phi^{L}( \mathit{x}_i))_+', (\Phi^{U}( \mathit{x}_i))_+' \}, \max\{(\Phi^{L}( \mathit{x}_i))_+', (\Phi^{U}( \mathit{x}_i))_+' \} \big]\\
    &= \big[ \min\{(\Phi^{L}( \mathit{x}_i))_-', (\Phi^{U}( \mathit{x}_i))_-' \}, \max\{(\Phi^{L}( \mathit{x}_i))_-', (\Phi^{U}( \mathit{x}_i))_-' \} \big].
        \end{aligned}
    \end{equation*}

    \item[(ii)] $(\Phi^{L}( \mathit{x}_i))_+'$ and $(\Phi^{U}( \mathit{x}_i))_+'$  exist. $\gamma_{\Phi^{L}}(t)$ and $\gamma_{\Phi^{U}}(t)$ are left complementary at 0, i.e., $C_{L(0)}(\gamma_{\Phi^{L}}) = C_{L(0)}(\gamma_{\Phi^{U}}) = \{ \mathit{k}^L, \mathit{k}^U \}$, where $\mathit{k}^L, \mathit{k}^U \in \mathbb{R}$ and $\mathit{k}^L < \mathit{k}^U$. Moreover,
    
   \begin{equation*}
   	  \begin{aligned}
\frac{\partial \tilde{h}}{\partial\mathit{x}_i} =\Phi'(\mathit{x}_i) &= \big[ \min\{(\Phi^{L}( \mathit{x}_i))_+', (\Phi^{U}( \mathit{x}_i))_+' \}, \max\{(\Phi^{L}( \mathit{x}_i))_+', (\Phi^{U}( \mathit{x}_i))_+' \} \big] \\
          &= [\mathit{k}^L, \mathit{k}^U]. 
            \end{aligned}    
\end{equation*}

    \item[(iii)] $(\Phi^{L}( \mathit{x}_i))_-'$ and $(\Phi^{U}( \mathit{x}_i))_-'$  exist. $\gamma_{\Phi^{L}}(t)$ and $\gamma_{\Phi^{U}}(t)$ are right complementary at 0, i.e., $C_{R(0)}(\gamma_{\Phi^{L}}) = C_{R(0)}(\gamma_{\Phi^{U}}) = \{ \mathit{k}^L, \mathit{k}^U \}$, where $\mathit{k}^L, \mathit{k}^U \in \mathbb{R}$ and $\mathit{k}^L < \mathit{k}^U$. Moreover,
    \begin{equation*}
        \begin{aligned}
          \frac{\partial \tilde{h}}{\partial\mathit{x}_i} =  \Phi'(\mathit{x}_i) &= \big[ \min\{(\Phi^{L}( \mathit{x}_i))_-', (\Phi^{U}( \mathit{x}_i))_-' \}, \max\{(\Phi^{L}( \mathit{x}_i))_-', (\Phi^{U}( \mathit{x}_i))_-' \} \big]\\
          & = [\mathit{k}^L, \mathit{k}^U].
        \end{aligned}
    \end{equation*}

    \item[(iv)] $\gamma_{\Phi^{L}}(t)$ and $\gamma_{\Phi^{U}}(t)$ are both left complementary and right complementary at 0, i.e., $C_{R(0)}(\gamma_{\Phi^{L}}) = C_{R(0)}(\gamma_{\Phi^{U}}) =C_{L(0)}(\gamma_{\Phi^{L}}) = C_{L(0)}(\gamma_{\Phi^{U}})= \{ \mathit{k}^L, \mathit{k}^U \}$, where $\mathit{k}^L, \mathit{k}^U \in \mathbb{R}$ and $\mathit{k}^L < \mathit{k}^U$. Moreover,
    \[
 \frac{\partial \tilde{h}}{\partial\mathit{x}_i} =  \Phi'(\mathit{x}_i) = [\mathit{k}^L, \mathit{k}^U].
    \]
\end{itemize}
\begin{proof}
    Since, the IVF $\Phi:\mathbb{R}\rightarrow\mathcal{I}(\mathbb{R})$ is defined as $
    \Phi(\mathit{\zeta} ) =\tilde{h}(\mathit{x} _1, \mathit{x} _2, \dots, \mathit{x}_{i-1}, \mathit{\zeta}, \mathit{x}_{i+1}, \dots, \mathit{x} _n).$
    The proof follows from (Theorem 2 in D. Qiu~\cite{D. Qiu}.)

\end{proof}

\begin{example} 
    Let $ \tilde{h} : \mathbb{R}^2 \to \mathcal{I}(\mathbb{R})$ defined by
\begin{equation*}
    \tilde{h}(\mathit{x},\mathit{y}) = [-|x| + y^2,\; |x| + y^2].
\end{equation*}
Then,
\begin{equation*}
    \tilde{h}^L(\mathit{x},\mathit{y}) = -|x| + y^2, \quad \tilde{h}^U(\mathit{x},\mathit{y}) = |x| + y^2
\end{equation*}
\noindent
Now, we compute the following:\vspace{0.2cm}\\
\noindent
\begin{equation*}
\left( \frac{\partial \tilde{h}^L}{\partial x} \right)_+(0, 0) 
= \lim_{t \to 0^+} \frac{\tilde{h}^L(0 + t, 0) - \tilde{h}^L(0, 0)}{t}
= \lim_{t \to 0^+} \frac{-|t| - 0}{t} = -1
\end{equation*}\vspace{0.2cm}
\noindent
\begin{equation*}
\left( \frac{\partial \tilde{h}^L}{\partial x} \right)_-(0, 0) 
= \lim_{t \to 0^-} \frac{\tilde{h}^L(t, 0) - \tilde{h}^L(0, 0)}{t}
= \lim_{t \to 0^-} \frac{-|t| - 0}{t} = 1\qquad
\end{equation*}
\noindent
\vspace{0.2cm}
Similarly we have, \vspace{0.2cm}
$\left( \frac{\partial \tilde{h}^U}{\partial x} \right)_+(0, 0) 
=1$ and $\left( \frac{\partial \tilde{h}^U}{\partial x} \right)_-(0, 0) 
=-1$

\vspace{0.3cm}
Thus, it is evident that the partial derivatives of the endpoint functions with respect to \(x\) at \((0,0)\) does not exist. Nevertheless,

\begin{equation*}
    \begin{aligned}
        &\left[
    \min \left\{
        \left( \frac{\partial \tilde{h}^L}{\partial x} \right)_{+}(0, 0),\,
        \left( \frac{\partial \tilde{h}^U}{\partial x} \right)_{+}(0, 0)
    \right\}, \right. \\
&\quad \left.
    \max \left\{
        \left( \frac{\partial \tilde{h}^L}{\partial x} \right)_{+}(0, 0),\,
        \left( \frac{\partial \tilde{h}^U}{\partial x} \right)_{+}(0, 0)
    \right\}
\right]
= [-1, 1]
    \end{aligned}
\end{equation*}

\begin{equation*}
    \begin{aligned}
        &\left[
    \min \left\{
        \left( \frac{\partial \tilde{h}^L}{\partial x} \right)_{-}(0, 0),\,
        \left( \frac{\partial \tilde{h}^U}{\partial x} \right)_{-}(0, 0)
    \right\}, \right. \\
&\quad \left.
    \max \left\{
        \left( \frac{\partial \tilde{h}^L}{\partial x} \right)_{-}(0, 0),\,
        \left( \frac{\partial \tilde{h}^U}{\partial x} \right)_{-}(0, 0)
    \right\}
\right]
= [-1, 1]
    \end{aligned}
\end{equation*}\vspace{0.2cm}\\
Therefore, by Theorem 3.1. case (i), $\frac{\partial \tilde{h}}{\partial x}(0, 0)$ exist and $\frac{\partial \tilde{h}}{\partial x}(0, 0)=[-1,1]$.\\
\end{example}
\begin{proposition}
     Suppose $\frac{\partial \tilde{h}^L}{\partial\mathit{x}_i}~ \textit{and}~\frac{\partial \tilde{h}^U}{\partial\mathit{x}_i}$ exist. Then, the gH-partial derivatives of $\tilde{h}$ exist and 
                        
             \begin{align*}          
                        \frac{\partial \tilde{h}}{\partial\mathit{x}_i}=\Phi'(\mathit{x}_i) = \big[  \min\{\frac{\partial \tilde{h}^L}{\partial\mathit{x}_i}, \frac{\partial \tilde{h}^U}{\partial\mathit{x}_i} \},\max\{\frac{\partial \tilde{h}^L}{\partial\mathit{x}_i}, \frac{\partial \tilde{h}^U}{\partial\mathit{x}_i} \} \big].
        \end{align*}

\end{proposition}
\noindent
Further, we demonstrate Theorem~\ref{theorem 3.1 p3} through the following examples.
\begin{example}
    Let \(  \tilde{h}: \mathbb{R}^2 \to \mathcal{I}(\mathbb{R}) \) be defined as
\begin{equation*}
     \tilde{h}(x, y) =\left[|x| + \sin x + y^2,|x| + \sin x + (x - 4)^2 + y^2\right]
\end{equation*}
Then, \(  \tilde{h}^L,  \tilde{h}^U : \mathbb{R}^2 \to \mathbb{R} \), the lower and upper end functions of \( \tilde{h}(\mathit{x},\mathit{y}) \) are respectively, 
\[
 \tilde{h}^L(x, y) = |x| + \sin x + y^2, \quad
 \tilde{h}^U(x, y) = |x| + \sin x + (x - 4)^2 + y^2.
\]
We now compute the following,
\begin{align*}
             \left( \frac{\partial  \tilde{h}^L}{\partial x} \right)_+(0, 0) 
&= \lim_{t \to 0^+} \frac{ \tilde{h}^L(t, 0) -  \tilde{h}^L(0, 0)}{t}\\
&= \lim_{t \to 0^+} \frac{|t| + \sin t}{t}\\
&=2  
\end{align*}\vspace{-1cm} 
 \begin{align*}
       \left( \frac{\partial  \tilde{h}^L}{\partial x} \right)_-(0, 0) 
&= \lim_{t \to 0^-} \frac{ \tilde{h}^L(t, 0) -  \tilde{h}^L(0, 0)}{t}\\
&= \lim_{t \to 0^-} \frac{|t| + \sin t}{t}\\
&=0 \\
    \end{align*}\vspace{-1cm}
    
Thus, right and left partial derivatives of $\tilde{h}^L$  w.r.t $x$ exist at (0,0).Also,\vspace{.6cm}\vspace{-1.5cm}\\

    \begin{align*}
        \left( \frac{\partial  \tilde{h}^U}{\partial x} \right)_+(0, 0) 
&= \lim_{t \to 0^+} \frac{ \tilde{h}^L(t, 0) -  \tilde{h}^L(0, 0)}{t}\\
&= \lim_{t \to 0^+} \frac{|t| + \sin t+(t-4)^2-16}{t}\\
&=-6
    \end{align*}
    \begin{align*}
        \left( \frac{\partial  \tilde{h}^U}{\partial x} \right)_-(0, 0) 
&= \lim_{t \to 0^-} \frac{ \tilde{h}^L(t, 0) -  \tilde{h}^L(0, 0)}{t}\\
&= \lim_{t \to 0^-} \frac{|t| + \sin t+(t-4)^2-16}{t}\\
&=-8
    \end{align*}

Thus, right and left partial derivatives of $\tilde{h}^U$ exist at (0,0). But\\
\begin{equation*}
    \begin{aligned}
        &\left[
    \min \left\{
        \left( \frac{\partial \tilde{h}^L}{\partial x} \right)_{+}(0, 0),\,
        \left( \frac{\partial \tilde{h}^U}{\partial x} \right)_{+}(0, 0)
    \right\}, \right. \\
&\quad \left.
    \max \left\{
        \left( \frac{\partial \tilde{h}^L}{\partial x} \right)_{+}(0, 0),\,
        \left( \frac{\partial \tilde{h}^U}{\partial x} \right)_{+}(0, 0)
    \right\}
\right]
= [-6, 2]
    \end{aligned}
\end{equation*}
\begin{equation*}
    \begin{aligned}
        &\left[
    \min \left\{
        \left( \frac{\partial \tilde{h}^L}{\partial x} \right)_{-}(0, 0),\,
        \left( \frac{\partial \tilde{h}^U}{\partial x} \right)_{-}(0, 0)
    \right\}, \right. \\
&\quad \left.
    \max \left\{
        \left( \frac{\partial \tilde{h}^L}{\partial x} \right)_{-}(0, 0),\,
        \left( \frac{\partial \tilde{h}^U}{\partial x} \right)_{-}(0, 0)
    \right\}
\right]
= [-8, 0].
    \end{aligned}
\end{equation*}\\[0.5cm]
Therefore, by Theorem 3.1. case (i), $\frac{\partial \tilde{h}}{\partial x}$ does not exist at (0,0).\\
\end{example}
\begin{example}\label{example 3.2}
    Let \( \tilde{h} : \mathbb{R}^2 \to \mathcal{I}(\mathbb{R}) \) be defined as:
\[
\tilde{h}(\mathit{x},\mathit{y}) = 
\begin{cases}
[x, 2x + 1 + |y|] & \text{if } x > 0 \\
[0, 1] & \text{if } x = 0, y = 0 \\
[x, x^2 + 2x + 1] & \text{if } x < 0,\, x \in \mathbb{Q} \\
[2x, x^2 + x + 1] & \text{if } x < 0,\, x \in \mathbb{Q}^c
\end{cases}
\]

Then \( \tilde{h}^L, \tilde{h}^U : \mathbb{R}^2 \to \mathbb{R} \), be given as follows:

\begin{equation*}
    \begin{aligned}
        \tilde{h}^L(\mathit{x},\mathit{y}) = 
\begin{cases}
x & \text{if } x > 0 \\
0 & \text{if } x = 0, y = 0 \\
x & \text{if } x < 0,\, x \in \mathbb{Q} \\
2x & \text{if } x < 0,\, x \in \mathbb{Q}^c
\end{cases}
\qquad
\tilde{h}^U(\mathit{x},\mathit{y}) = 
\begin{cases}
2x + 1 + |y| & \text{if } x > 0 \\
1 & \text{if } x = 0, y = 0 \\
x^2 + 2x + 1 & \text{if } x < 0,\, x \in \mathbb{Q} \\
x^2 + x + 1 & \text{if } x < 0,\, x \in \mathbb{Q}^c
\end{cases}
    \end{aligned}
    \end{equation*}
Now, $\gamma_{\tilde{h}^L}(t) ~\text{and}~ \gamma_{\tilde{h}^U}(t)$ are:

\begin{equation*}
    \begin{aligned}
        \gamma_{\tilde{h}^L}(t) = \frac{\tilde{h}^L(x + t, y) - \tilde{h}^L(\mathit{x},\mathit{y})}{t}, \quad
\gamma_{\tilde{h}^U}(t) = \frac{\tilde{h}^U(x + t, y) - \tilde{h}^U(\mathit{x},\mathit{y})}{t}, \quad t \in \mathbb{R} \setminus \{0\}
    \end{aligned}
\end{equation*}

\begin{equation*}
    \begin{aligned}
        \left( \frac{\partial \tilde{h}^L}{\partial x} \right)_+(0, 0)
&= \lim_{t \to 0^+} \frac{\tilde{h}^L(t, 0) - \tilde{h}^L(0, 0)}{t}\\
&= \lim_{t \to 0^+} \frac{t - 0}{t}\\
&= 1\\
    \end{aligned}
\end{equation*}

\begin{equation*}
    \begin{aligned}
        \left( \frac{\partial \tilde{h}^U}{\partial x} \right)_+(0, 0)
&= \lim_{t \to 0^+} \frac{\tilde{h}^U(t, 0) - \tilde{h}^U(0, 0)}{t}\\
&= \lim_{t \to 0^+} \frac{2t+1 - 1}{t}\\
&= 2\\ 
    \end{aligned}
\end{equation*}
Thus, right partial derivatives of $\tilde{h}^L$ and $\tilde{h}^U$ exist at (0,0). Now,
\begin{equation*}
 \begin{aligned}
 \left( \frac{\partial \tilde{h}^L}{\partial x} \right)_{-} (0,0) &= \lim_{t \to 0^-} \frac{\tilde{h}^L(t,0) - \tilde{h}^L(0,0)}{t}\\
        &= \lim_{t \to 0^-}
\begin{cases}
\frac{t - 0}{t} , & \text{if } t \in \mathbb{Q} \\
\frac{2t - 0}{t} , & \text{if } t \in \mathbb{Q}^c
\end{cases}\\
&= \begin{cases}
1, & \text{if } t \in \mathbb{Q} \\
2, & \text{if } t \in \mathbb{Q}^c
\end{cases}\\
\left( \frac{\partial \tilde{h}^U}{\partial x} \right)_{-} (0,0) &= \lim_{t \to 0^-} \frac{\tilde{h}^U(t,0) - \tilde{h}^U(0,0)}{t}\\
        &= \lim_{t \to 0^-}
\begin{cases}
\frac{t^2 + 2t + 1 - 1}{t} = \frac{t^2 + 2t}{t}, & \text{if } t \in \mathbb{Q} \\
\frac{t^2 + t + 1 - 1}{t} = \frac{t^2 + t}{t}, & \text{if } t \in \mathbb{Q}^c
\end{cases}\\
&= \begin{cases}
2, & \text{if } t \in \mathbb{Q} \\
1, & \text{if } t \in \mathbb{Q}^c
\end{cases}
    \end{aligned}
\end{equation*}

Therefore,
\begin{equation*}
    C_{L(0)}(\gamma_{\tilde{h}^L}) = C_{L(0)}(\gamma_{\tilde{h}^U}) = \{1, 2\} = \{\mathit{k}^L, \mathit{k}^U\}, \quad \text{where } \mathit{k}^L \leq \mathit{k}^U
\end{equation*}
and
\[
\lim_{t \to 0^-} \min \left\{ \gamma_{\tilde{h}^L}(t), \gamma_{\tilde{h}^U}(t) \right\} = 1, \quad 
\lim_{t \to 0^-} \max \left\{ \gamma_{\tilde{h}^L}(t), \gamma_{\tilde{h}^U}(t) \right\} = 2.
\]
\\
\noindent Therefore, by definition 3.2, $\gamma_{\tilde{h}^L}(t) ~\&~ \gamma_{\tilde{h}^U}(t)$ are  \textbf{left complementary} at 0. Thus,\vspace{0.2cm}\\
\begin{equation*}
    \begin{aligned}
        &\left[
    \min \left\{
        \left( \frac{\partial \tilde{h}^L}{\partial x} \right)_{+}(0, 0),\,
        \left( \frac{\partial \tilde{h}^U}{\partial x} \right)_{+}(0, 0)
    \right\}, \right. \\
&\quad \left.
    \max \left\{
        \left( \frac{\partial \tilde{h}^L}{\partial x} \right)_{+}(0, 0),\,
        \left( \frac{\partial \tilde{h}^U}{\partial x} \right)_{+}(0, 0)
    \right\}
\right]
= [1, 2]=[\mathit{k}^L,\mathit{k}^U].
    \end{aligned}
\end{equation*}

Therefore, by Theorem 3.1. case (ii), $\frac{\partial \tilde{h}}{\partial x}(0, 0)$ exist and $\frac{\partial \tilde{h}}{\partial x}(0, 0)=[1,2]$.\\
\end{example}
\begin{example}
    Let \( \tilde{h} : \mathbb{R}^2 \to \mathcal{I}(\mathbb{R}) \) be defined as:
\[
\tilde{h}(\mathit{x},\mathit{y}) = 
\begin{cases}
[x, 2x  + |y|] & \text{if } x > 0 \\
[0, 0] & \text{if } x = 0, y = 0 \\
[x, x + |y| ] & \text{if } x < 0,\, x \in \mathbb{Q} \\
[2x, 2x +|y|] & \text{if } x < 0,\, x \in \mathbb{Q}^c
\end{cases}
\]

Then \( \tilde{h}^L, \tilde{h}^U : \mathbb{R}^2 \to \mathbb{R} \), be given as follows:

\begin{equation*}
    \begin{aligned}
        \tilde{h}^L(\mathit{x},\mathit{y}) = 
\begin{cases}
x & \text{if } x > 0 \\
0 & \text{if } x = 0, y = 0 \\
x & \text{if } x < 0,\, x \in \mathbb{Q} \\
2x & \text{if } x < 0,\, x \in \mathbb{Q}^c
\end{cases}
\qquad
\tilde{h}^U(\mathit{x},\mathit{y}) = 
\begin{cases}
2x  + |y| & \text{if } x > 0 \\
0 & \text{if } x = 0, y = 0 \\
x+|y| & \text{if } x < 0,\, x \in \mathbb{Q} \\
2x+|y| & \text{if } x < 0,\, x \in \mathbb{Q}^c
\end{cases}
    \end{aligned}
    \end{equation*}
Now, $\gamma_{\tilde{h}^L}(t) ~\text{and}~ \gamma_{\tilde{h}^U}(t)$ are:

\begin{equation*}
    \begin{aligned}
        \gamma_{\tilde{h}^L}(t) = \frac{\tilde{h}^L(x + t, y) - \tilde{h}^L(\mathit{x},\mathit{y})}{t}, \quad
\gamma_{\tilde{h}^U}(t) = \frac{\tilde{h}^U(x + t, y) - \tilde{h}^U(\mathit{x},\mathit{y})}{t}, \quad t \in \mathbb{R} \setminus \{0\}
    \end{aligned}
\end{equation*}

\begin{equation*}
    \begin{aligned}
        \left( \frac{\partial \tilde{h}^L}{\partial x} \right)_+(0, 0)
&= \lim_{t \to 0^+} \frac{\tilde{h}^L(t, 0) - \tilde{h}^L(0, 0)}{t}\\
&= \lim_{t \to 0^+} \frac{t - 0}{t}\\
&= 1\\
    \end{aligned}
\end{equation*}

\begin{equation*}
    \begin{aligned}
        \left( \frac{\partial \tilde{h}^U}{\partial x} \right)_+(0, 0)
&= \lim_{t \to 0^+} \frac{\tilde{h}^U(t, 0) - \tilde{h}^U(0, 0)}{t}\\
&= \lim_{t \to 0^+} \frac{2t - 0}{t}\\
&= 2\\
    \end{aligned}
\end{equation*}
Thus, right partial derivative of $\tilde{h}^L$ and $\tilde{h}^U$  exist at (0,0). Now,
\begin{equation*}
    \begin{aligned}
        \left( \frac{\partial \tilde{h}^L}{\partial x} \right)_{-} (0,0) &= \lim_{t \to 0^-} \frac{\tilde{h}^L(t,0) - \tilde{h}^L(0,0)}{t}\\
        &= \lim_{t \to 0^-}
\begin{cases}
\frac{t - 0}{t} , & \text{if } t \in \mathbb{Q} \\
\frac{2t - 0}{t} , & \text{if } t \in \mathbb{Q}^c
\end{cases}\\
&= \begin{cases}
1, & \text{if } t \in \mathbb{Q} \\
2, & \text{if } t \in \mathbb{Q}^c
\end{cases}
    \end{aligned}
\end{equation*}

\begin{equation*}
    \begin{aligned}
        \left( \frac{\partial \tilde{h}^U}{\partial x} \right)_{-} (0,0) &= \lim_{t \to 0^-} \frac{\tilde{h}^U(t,0) - \tilde{h}^U(0,0)}{t}\\
        &= \lim_{t \to 0^-}
\begin{cases}
\frac{t - 0}{t} , & \text{if } t \in \mathbb{Q} \\
\frac{2t-0}{t} , & \text{if } t \in \mathbb{Q}^c
\end{cases}\\
&= \begin{cases}
1, & \text{if } t \in \mathbb{Q} \\
2, & \text{if } t \in \mathbb{Q}^c
\end{cases}
    \end{aligned}
\end{equation*}

Therefore,
\begin{equation*}
    C_{L(0)}(\gamma_{\tilde{h}^L}) = C_{L(0)}(\gamma_{\tilde{h}^U}) = \{1, 2\} = \{\mathit{k}^L, \mathit{k}^U\}, \quad \text{where } \mathit{k}^L \leq \mathit{k}^U
\end{equation*}
and
\begin{equation*}
    \begin{aligned}
        \lim_{t \to 0^-} \min \left\{ \gamma_{\tilde{h}^L}(t), \gamma_{\tilde{h}^U}(t) \right\} = \begin{cases}
1, & \text{if } t \in \mathbb{Q} \\
2, & \text{if } t \in \mathbb{Q}^c
\end{cases},
    \end{aligned}
\end{equation*}
\begin{equation*}
    \begin{aligned}
        \lim_{t \to 0^-} \max \left\{ \gamma_{\tilde{h}^L}(t), \gamma_{\tilde{h}^U}(t) \right\}  = \begin{cases}
1, & \text{if } t \in \mathbb{Q} \\
2, & \text{if } t \in \mathbb{Q}^c
\end{cases}
    \end{aligned}
\end{equation*}\vspace{-0.2cm}\\

\noindent Therefore, by definition 3.2, $\gamma_{\tilde{h}^L}(t)~\&~ \gamma_{\tilde{h}^U}(t)$ are   \textbf{not left complementary} at 0. Thus, by Theorem 3.1. case(ii), $\left( \frac{\partial \tilde{h}}{\partial x} \right)_{(x,y)=(0,0)} $ does not exit. Nevertheless,\vspace{0.2cm}\\
\begin{equation*}
    \begin{aligned}
       & \left[
            \min \left\{ 
                \left( \frac{\partial \tilde{h}^L}{\partial x} \right)_+ (0,0),\,
                \left( \frac{\partial \tilde{h}^U}{\partial x} \right)_+ (0,0) 
            \right\}, \right. \\
        &\quad \left.
            \max \left\{ 
                \left( \frac{\partial \tilde{h}^L}{\partial x} \right)_+ (0,0),\,
                \left( \frac{\partial \tilde{h}^U}{\partial x} \right)_+ (0,0) 
            \right\}
        \right]\\
        &=[1,2]=[\mathit{k}^L, \mathit{k}^U]
    \end{aligned}
\end{equation*}
\end{example}\vspace{0.2cm}
\begin{example}
     Let \( \tilde{h} : \mathbb{R}^2 \to \mathcal{I}(\mathbb{R}) \) be defined as:
\[
\tilde{h}(\mathit{x},\mathit{y}) = 
\begin{cases}
[x, 2x + 1 + |y|] & \text{if } x < 0 \\
[0, 1] & \text{if } x = 0, y = 0 \\
[x, x^2 + 2x + 1] & \text{if } x > 0,\, x \in \mathbb{Q} \\
[2x, x^2 + x + 1] & \text{if } x > 0,\, x \in \mathbb{Q}^c
\end{cases}
\]

Then \( \tilde{h}^L, \tilde{h}^U : \mathbb{R}^2 \to \mathbb{R} \), be given as follows:

\begin{equation*}
    \begin{aligned}
        \tilde{h}^L(\mathit{x},\mathit{y}) = 
\begin{cases}
x & \text{if } x < 0 \\
0 & \text{if } x = 0, y = 0 \\
x & \text{if } x > 0,\, x \in \mathbb{Q} \\
2x & \text{if } x > 0,\, x \in \mathbb{Q}^c
\end{cases}
\quad
\tilde{h}^U(\mathit{x},\mathit{y}) = 
\begin{cases}
2x + 1 + |y| & \text{if } x < 0 \\
1 & \text{if } x = 0, y = 0 \\
x^2 + 2x + 1 & \text{if } x > 0,\, x \in \mathbb{Q} \\
x^2 + x + 1 & \text{if } x >0,\, x \in \mathbb{Q}^c
\end{cases}
    \end{aligned}
    \end{equation*}
  By Theorem 3.1. case (iii),~$\frac{\partial \tilde{h}}{\partial x}(0, 0)$ exist and $\frac{\partial \tilde{h}}{\partial x}(0, 0)=[1,2]$.\\
\end{example}
\begin{example}
     Let \( \tilde{h} : \mathbb{R}^2 \to \mathcal{I}(\mathbb{R}) \) be defined as:
\[
\tilde{h}(\mathit{x}, \mathit{y}) = 
\begin{cases}
[x, 2x^2 + 2x + 1+ |y|] & \text{if } x > 0,\, x \in \mathbb{Q}\\
[2x, 2x^2 +x+ 1 + |y|] & \text{if } x > 0\, x \in \mathbb{Q}^c \\
[0, 1] & \text{if } x = 0, y = 0 \\
[x, x^2 + 2x + 1] & \text{if } x < 0,\, x \in \mathbb{Q} \\
[2x, x^2 + x + 1] & \text{if } x < 0,\, x \in \mathbb{Q}^c
\end{cases}
\]

Then \( \tilde{h}^L, \tilde{h}^U : \mathbb{R}^2 \to \mathbb{R} \), be given as follows:

\begin{equation*}
    \begin{aligned}
        \tilde{h}^L(\mathit{x},\mathit{y}) = 
\begin{cases}
x & \text{if } x > 0,\, x \in \mathbb{Q} \\
2x & \text{if } x > 0,\, x \in \mathbb{Q}^c \\
0 & \text{if } x = 0,\, y = 0 \\
x & \text{if } x < 0,\, x \in \mathbb{Q} \\
2x & \text{if } x < 0,\, x \in \mathbb{Q}^c
\end{cases}
\quad
\tilde{h}^U(\mathit{x},\mathit{y}) = 
\begin{cases}
2x^2 +2x+ 1 + |y| & \text{if } x > 0,\,x \in \mathbb{Q}  \\
2x^2 +x+ 1 + |y| & \text{if } x > 0,\,x \in \mathbb{Q}^c  \\
1 & \text{if } x = 0, y = 0 \\
x^2 + 2x + 1 & \text{if } x < 0,\, x \in \mathbb{Q} \\
x^2 + x + 1 & \text{if } x < 0,\, x \in \mathbb{Q}^c
\end{cases}
    \end{aligned}
    \end{equation*}
By Theorem 3.1. case (iv),~$\frac{\partial \tilde{h}}{\partial x}(0, 0)$ exist and $\frac{\partial \tilde{h}}{\partial x}(0, 0)=[1,2]$.
\end{example}
\vspace{-0.76cm}
\begin{center}
\item  \subsection{\textit{gH-product of a vector with n-tuples of intervals}}   
\end{center}
Next, we introduce the gH-product of a vector with n-tuples of intervals. 
  Ghosh et. al.~\cite{Ghosh2020}, defined  the product \( \nu \cdot \tilde{\mathcal{K}} \) as follows,
\[
\nu \cdot \tilde{\mathcal{K}} = \sum_{i=1}^{n} \nu_i \mathcal{K}_i 
\]
where, $\nu=(\nu_1,\nu_2,.....,\nu_n)\in \mathbb{R}^n$ and $\tilde{\mathcal{K}}=(\mathcal{K}_1,\mathcal{K}_2,....,\mathcal{K}_n)\in\mathcal{I}^n(\mathbb{R})$.\\
In particular, for n=2, we have $\nu=(\nu_1,\nu_2)\in \mathbb{R}^2$ and the interval $\tilde{\mathcal{K}}=(\mathcal{K}_1,\mathcal{K}_2)\in\mathcal{I}^2(\mathbb{R})$.
Now using Ghosh~\cite{Ghosh2020} definition, we have \\
\[
\nu \cdot \tilde{\mathcal{K}} = \sum_{i=1}^{2} \nu_i \mathcal{K}_i =\nu_1 \mathcal{K}_1+\nu_2 \mathcal{K}_2
\]
For instance, let \( \nu = (1, -1) \), $\mathcal{K}_1=[1,2]$ and \( \tilde{\mathcal{K}} = (\mathcal{K}_1,\mathcal{K}_1) \). Then,
\begin{align*}
\nu \cdot \tilde{\mathcal{K}} &= 1 \cdot [1, 2] + (-1) \cdot [1, 2] \\
&= [1, 2] + [-2, -1] \\
&= [-1, 1]\\
&\neq [0,0].
\end{align*}
It is to be noted that the above expression by the definition of Ghosh~\cite{Ghosh2020} is equivalent to the Minkowski difference of two intervals i.e. $\nu.\tilde{\mathcal{K}}=\mathcal{K}_1-\mathcal{K}_1\neq0$.
To overcome Minkowski difference, the Hukuhara difference (H-difference) and later the gH-difference were introduced by \cite{Hukuhara,Stefanini2009}.\vspace{0.2cm}\\
Now, we define the gH-product of a vector with n-tuples of intervals as follows:\vspace{0.2cm}\\
Let  \( \nu = (\nu_1, \nu_2, \ldots, \nu_n) \in \mathbb{R}^n \) , \( \tilde{\mathcal{K}} = (\mathcal{K}_1, \mathcal{K}_2, \ldots, \mathcal{K}_n) \in \mathcal{I}^n(\mathbb{R}) \), 
$j^+ = \{ i : \nu_i \geq 0 \} $  and $j^- = \{ i : \nu_i < 0 \}.$\\
Now, $\bm{\langle ,  \rangle_{gH}: \mathbb{R}^n \times \mathcal{I}^n(\mathbb{R}) \longrightarrow \mathcal{I} (\mathbb{R})}$, be defined as: 
\begin{equation*} 
\begin{aligned}
\bm{\langle \nu}, \bm{\tilde{\mathcal{K}} \rangle_{\textbf{gH}}} &= \bm{\sum_{i \in j^+} \nu_i \mathcal{K}_i} \,\bm{\ominus_{gH}}\,\bm{ \sum_{k \in j^-} |\nu_k| \mathcal{K}_k} \\
&= \left[ \sum_{i \in j^+} \nu_i k_i^L , \sum_{i \in j^+} \nu_i k_i^U \right]
\,\ominus_{gH}\,
\left[ \sum_{k \in j^-} |\nu_k| k_k^L , \sum_{k \in j^-} |\nu_k| k_k^U \right] \\
&= \left[ \min \{ p, q \},\, \max \{ p, q \} \right], \\
&\text{where} \quad
p = \sum_{i \in j^+} \nu_i k_i^L - \sum_{k \in j^-} |\nu_k| k_k^L, \quad
q = \sum_{i \in j^+} \nu_i k_i^U - \sum_{k \in j^-} |\nu_k| k_k^U.
\end{aligned}
\end{equation*}

\textbf{Case 1:} \( p \leq q \)

\begin{equation*}
\begin{aligned}
\langle \nu, \tilde{\mathcal{K}} \rangle_{gH}
&= \left[
\sum_{i \in j^+} \nu_i k_i^L - \sum_{k \in j^-} |\nu_k| k_k^L,\,
\sum_{i \in j^+} \nu_i k_i^U - \sum_{k \in j^-} |\nu_k| k_k^U
\right] \\[-4pt]
&= [ \nu  \mathit{k}^L,\; \nu  \mathit{k}^U ],
\end{aligned}
\end{equation*}

where
\[
\mathit{k}^L = (k_1^L, k_2^L, \ldots, k_n^L), \quad
\mathit{k}^U = (k_1^U, k_2^U, \ldots, k_n^U).
\]

\textbf{Case 2:} \( p > q \)

\[
\langle \nu, \tilde{\mathcal{K}} \rangle_{gH} = [ \nu  \mathit{k}^U,\, \nu \mathit{k}^L ].
\]\vspace{0.2cm}\\
\textit{The value $\langle \nu, \tilde{\mathcal{K}} \rangle_{gH}$ is called gH-product of $\nu$ with $\tilde{\mathcal{K}}$.}
\vspace{0.2cm}\\ 
Hence, for n=2, the gH-product $\langle \nu , \tilde{\mathcal{K}} \rangle_{gH}$ is equivalent with the gH-difference and $\nu \cdot \tilde{\mathcal{K}}$ is equivalent with Minkowski difference.\\
\textbf{Note:} 
\begin{enumerate}
    \item If all the component of $\nu$ are positive then $\langle \nu, \tilde{\mathcal{K}} \rangle_{gH}$ is given as
\begin{equation*}
    \begin{aligned}
        \langle \nu, \tilde{\mathcal{K}} \rangle_{gH}=\sum_{i=1}^{n} \nu_i \mathcal{K}_i=\sum_{i=1}^{n} \nu_i \mathcal{K}_i\ominus_{gH}[0,0]
    \end{aligned}
\end{equation*}
    \item If all the component of $\nu$ are negative then $\langle \nu, \tilde{\mathcal{K}} \rangle_{gH}$ is given as
\begin{equation*}
    \begin{aligned}
        \langle \nu, \tilde{\mathcal{K}} \rangle_{gH}=\sum_{i=1}^{n} \nu_i \mathcal{K}_i=[0,0]\ominus_{gH}\sum_{i=1}^{n} |\nu_i| \mathcal{K}_i
    \end{aligned}
\end{equation*}

\end{enumerate}

\begin{remark}
    If each component of n-tuples of the interval  \( \tilde{\mathcal{K}} \in \mathcal{I}^n(\mathbb{R}) \) is  a \emph{degenerate interval} (i.e., \( \mathcal{K}_i = [k^i, k^i] \) for some \( k^i \in \mathbb{R} \), see Moore~\cite{Moore1966}),then $gH$ product coincides with the dot product.
\end{remark}
\begin{proposition}
  Let \( \nu = (\nu_1, \nu_2, \ldots, \nu_n),\omega=(\omega_1, \omega_2, \ldots, \omega_n) \in \mathbb{R}^n \), $\tilde{\lambda} \in \mathbb{R}$ \text{and} \( \tilde{\mathcal{K}} = (\mathcal{K}_1, \mathcal{K}_2, \ldots, \mathcal{K}_n) \in \mathcal{I}^n(\mathbb{R}) \). Then:\\
  \begin{itemize}
 \label{Prop 3.2.1}\item[(i)] $\langle -\nu, \tilde{\mathcal{K}} \rangle_{gH}=-\langle \nu, \tilde{\mathcal{K}} \rangle_{gH}$  \\
 \item[(ii)]  $\langle \tilde{\lambda}\nu, \tilde{\mathcal{K}} \rangle_{gH}=\tilde{\lambda}\langle \nu, \tilde{\mathcal{K}} \rangle_{gH}$\\
 \item[(iii)]  In general,$\langle \nu + \omega, \tilde{\mathcal{K}} \rangle_{gH} \neq \langle \nu, \tilde{\mathcal{K}} \rangle_{gH} + \langle \omega, \tilde{\mathcal{K}} \rangle_{gH}$\\
 \item[(iv)]  Suppose \( \nu \neq 0 \), then \( \langle \nu, \tilde{\mathcal{K}} \rangle_{gH} = 0 \Leftrightarrow \nu \perp \mathit{k}^L~ \&~~ \nu \perp \mathit{k}^U \)
  \end{itemize}
\end{proposition}
\begin{proof}
\begin{itemize}
    \item[\textit{(i)}]  \mbox{By the definition of $gH$-product, we have}\vspace{0.2cm}
    \begin{align*}
    \langle \nu, \tilde{\mathcal{K}} \rangle_{gH} = &\sum_{i \in j^+} \nu_i \mathcal{K}_i \,\ominus_{gH}\, \sum_{k \in j^-} |\nu_k| \mathcal{K}_k \\
= &\left[ \sum_{i \in j^+} \nu_i k_i^L , \sum_{i \in j^+} \nu_i k_i^U \right]
\,\ominus_{gH}\,
\left[ \sum_{k \in j^-} |\nu_k| k_k^L , \sum_{k \in j^-} |\nu_k| k_k^U \right] \\
= &\left[ \min \{ p, q \},\, \max \{ p, q \} \right], 
\end{align*}
\text{where} \quad
$p = \sum_{i \in j^+} \nu_i k_i^L - \sum_{k \in j^-} |\nu_k| k_k^L$, \quad
$q = \sum_{i \in j^+} \nu_i k_i^U - \sum_{k \in j^-} |\nu_k| k_k^U.$
Now, we have 
  \begin{equation*}
      \begin{aligned}
          \langle -\nu, \tilde{\mathcal{K}} \rangle_{gH}&= \sum_{k \in j^-} |\nu_k| \mathcal{K}_k \,\ominus_{gH}\, \sum_{i \in j^+} \nu_i \mathcal{K}_i\\
          &=\left[ \sum_{k \in j^-} |\nu_k| k_k^L , \sum_{k \in j^-} |\nu_k| k_k^U \right] 
\,\ominus_{gH}\,\left[ \sum_{i \in j^+} \nu_i k_i^L , \sum_{i \in j^+} \nu_i k_i^U \right]
 \\
&= \left[ \min \{ p^*, q^* \},\, \max \{ p^*, q^* \} \right], \\
      \end{aligned}
  \end{equation*}  
  where,
 \begin{equation*}
\begin{aligned}
     p^* &= \sum_{k \in j^-} |\nu_k| k_k^L - \sum_{i \in j^+} \nu_i k_i^L=-p,\\
     q^* &= \sum_{k \in j^-} |\nu_k| k_k^U - \sum_{i \in j^+} \nu_i k_i^U=-q.
\end{aligned}
 \end{equation*}
 \begin{itemize}
     \item [Case (a):]  When $p^*\leq q^* ~i.e.-p\leq-q \Rightarrow q\leq p$
  \begin{equation*}
      \begin{aligned}
           \langle -\nu, \tilde{\mathcal{K}} \rangle_{gH}&=[\sum_{k \in j^-} |\nu_k| k_k^L - \sum_{i \in j^+} \nu_i k_i^L,\sum_{k \in j^-} |\nu_k| k_k^U - \sum_{i \in j^+} \nu_i k_i^U]\\
          & =[-p,-q]\\
          &=-[q,p]\\
          &=-\langle \nu, \tilde{\mathcal{K}} \rangle_{gH}
      \end{aligned}
  \end{equation*}
  
  \item [Case (b):] When $p^*> q^*~~ i.e. -p>-q\Rightarrow q>p$ 
  \begin{equation*}
      \begin{aligned}
           \langle -\nu, \tilde{\mathcal{K}} \rangle_{gH}&=[\sum_{k \in j^-} |\nu_k| k_k^U - \sum_{i \in j^+} \nu_i k_i^U,\sum_{k \in j^-} |\nu_k| k_k^L - \sum_{i \in j^+} \nu_i k_i^L]\\
          & =[-q,-p]\\
          &=-[p,q]\\
          &=-\langle \nu, \tilde{\mathcal{K}} \rangle_{gH}
      \end{aligned}
  \end{equation*}
  \end{itemize}
 
\noindent

\item[\textit{(ii)}] Let $\tilde{\lambda} \geq0$, \\
\begin{equation*}
    \begin{aligned}
      \langle \tilde{\lambda}\nu, \tilde{\mathcal{K}} \rangle_{gH} =& \sum_{i \in j^+} \tilde{\lambda}\nu_i \mathcal{K}_i \,\ominus_{gH}\, \sum_{k \in j^-} |\tilde{\lambda}\nu_k| \mathcal{K}_k \\
     =&\tilde{\lambda} \sum_{i \in j^+} \nu_i \mathcal{K}_i \,\ominus_{gH}\, \tilde{\lambda}\sum_{k \in j^-} |\nu_k| \mathcal{K}_k \\
     =& \tilde{\lambda}(\sum_{i \in j^+} \nu_i \mathcal{K}_i \,\ominus_{gH}\, \sum_{k \in j^-} |\nu_k| \mathcal{K}_k)\hspace{1cm} (\text{using Proposition \ref{prop2.1}}),\\
 =& \tilde{\lambda} \langle \nu, \tilde{\mathcal{K}} \rangle_{gH}  
    \end{aligned}\\
\end{equation*}
  Now, let $\tilde{\lambda}<0 \Rightarrow \tilde{\lambda}=-\mu$,where $ \mu\geq0$. 
  \begin{equation*}     
    \begin{aligned}
     \langle \tilde{\lambda}\nu, \tilde{\mathcal{K}} \rangle_{gH} =&  \langle -\mu\nu, \tilde{\mathcal{K}} \rangle_{gH}\\
     =& \mu \langle -\nu, \tilde{\mathcal{K}} \rangle_{gH}  \\
 =&  -\mu \langle \nu, \tilde{\mathcal{K}} \rangle_{gH}\hspace{3.0 cm}\text{(using Proposition \ref{Prop 3.2.1}.($\mathit{i}$))}, \\
 =& \tilde{\lambda} \langle \nu, \tilde{\mathcal{K}} \rangle_{gH} \\
    \end{aligned}
    \end{equation*}
Hence,$\langle \tilde{\lambda}\nu, \tilde{\mathcal{K}} \rangle_{gH}=\tilde{\lambda}\langle \nu, \tilde{\mathcal{K}} \rangle_{gH}$~~ for every $\tilde{\lambda}\in\mathbb{R}.$\\
\noindent
\item[\textit{(iii)}]
\textbf {Non-linearity of gH-Product in its first component}: In general, the linearity in its first component does not hold, i.e. 
    $$\langle \nu + \omega, \tilde{\mathcal{K}} \rangle_{gH} \neq \langle \nu, \tilde{\mathcal{K}} \rangle_{gH} + \langle \omega, \tilde{\mathcal{K}} \rangle_{gH}.$$\vspace{0.2cm}
This is demonstrated in the next example.

\begin{example}
    Let  
$
\nu = (1, -1), \omega = (-5, 4)$ and $\tilde{\mathcal{K}} = (\mathcal{K}_1, \mathcal{K}_2),  \\
\text{where}~ \mathcal{K}_1 = [1, 2],~ \mathcal{K}_2 = [3, 6]
$

Then,  
$
\nu + \omega = (-4, 3)
$

Now, compute:
\begin{align*}
\langle \nu + \omega, \tilde{\mathcal{K}} \rangle_{gH} &= \left[\min\{5, 10\}, \max\{5, 10\}\right] = [5, 10] \\
\langle \nu, \tilde{\mathcal{K}} \rangle_{gH} &= \left[\min\{-2, -4\}, \max\{-2, -4\}\right] = [-4, -2]\\
\langle \omega, \tilde{\mathcal{K}} \rangle_{gH} &= \left[\min\{7, 14\}, \max\{7, 14\}\right] = [7, 14]
\end{align*}
This demonstrates that:
$
\langle \nu + \omega, \tilde{\mathcal{K}} \rangle_{gH} \neq \langle \nu, \tilde{\mathcal{K}} \rangle_{gH} + \langle \omega, \tilde{\mathcal{K}} \rangle_{gH} $.
\end{example}\vspace{0.2cm}
However, the linearity can hold under the following assumptions.\\
The gH-product of \( \nu + \omega \) with \( \tilde{\mathcal{K}} \) is defined by:
\[
\langle \nu + \omega, \tilde{\mathcal{K}} \rangle_{gH} = \sum_{i \in j^+} (\nu_i + \omega_i) \mathcal{K}_i \ominus_{gH} \sum_{k \in j^-} |\nu_k + \omega_k| \mathcal{K}_k.
\]
Where,
\begin{align*}
p &= \sum_{i \in j^+} (\nu_i + \omega_i) k_i^L - \sum_{k \in j^-} |\nu_k + \omega_k| k_k^U \\
q &= \sum_{i \in j^+} (\nu_i + \omega_i) k_i^U - \sum_{k \in j^-} |\nu_k + \omega_k| k_k^L
\end{align*}

\noindent
Thus,
\begin{align*}    
\langle \nu + \omega, \tilde{\mathcal{K}} \rangle_{gH}& = \left[ \min\{p, q\},\max\{p, q\} \right]\\
&=\Big[ \min\left\{(\nu + \omega) \cdot \mathit{k}^L, (\nu + \omega) \cdot \mathit{k}^U \right\},\\
&\quad\qquad\max\left\{(\nu + \omega) \cdot \mathit{k}^L, (\nu + \omega) \cdot \mathit{k}^U \right\}\Big]
\end{align*}
 where, $\mathit{k}^L = (k_1^L, k_2^L, \dots, k_n^L)$,~$\mathit{k}^U =(k_1^U, k_2^U, \dots, k_n^U)$.\\
Now, if \( p < q \), then
$\langle \nu + \omega, A \rangle_{gH} = [p, q] = \left[(\nu + \omega) \cdot \mathit{k}^L, (\nu + \omega) \cdot \mathit{k}^U\right]$.\\
\vspace{0.2cm}
\noindent
If \( p > q \), then
$\langle \nu + \omega, \tilde{\mathcal{K}} \rangle_{gH} = [q, p] = \left[(\nu + \omega) \cdot \mathit{k}^U, (\nu + \omega) \cdot \mathit{k}^L\right]$\\
We also observe that,
\begin{align*}
\langle \nu, \tilde{\mathcal{K}} \rangle_{gH} + \langle w, \tilde{\mathcal{K}} \rangle_{gH} 
&= \left[ \min\{\nu \mathit{k}^L, \nu \mathit{k}^U\}, \max\{\nu \mathit{k}^L, \nu \mathit{k}^U\} \right] \\
&\quad + \left[ \min\{w \mathit{k}^L, w \mathit{k}^U\}, \max\{w \mathit{k}^L, w \mathit{k}^U\} \right]
\end{align*}
Consider the following  cases:\\
\begin{itemize}
\item [Case (a):] Suppose \( \nu \mathit{k}^L \leq \nu \mathit{k}^U \) and \( w \mathit{k}^L \leq w \mathit{k}^U\). Then,
\begin{align*}
\nu \mathit{k}^L + w \mathit{k}^L &\leq \nu \mathit{k}^U + w \mathit{k}^U \\
\Rightarrow (\nu + w) \mathit{k}^L &\leq (\nu + w) \mathit{k}^U
\end{align*}

Thus,
\begin{align*}
\langle \nu, \tilde{\mathcal{K}} \rangle_{gH} + \langle w, \tilde{\mathcal{K}} \rangle_{gH} 
&= [\nu \mathit{k}^L, \nu \mathit{k}^U] + [w \mathit{k}^L, w \mathit{k}^U] \\
&= [(\nu + w) \mathit{k}^L, (\nu + w) \mathit{k}^U] \\
&= \langle \nu + w, \tilde{\mathcal{K}} \rangle_{gH}
\end{align*}

\item [Case (b):] Analogously, when \( \nu \mathit{k}^L \geq \nu \mathit{k}^U \) and \( w \mathit{k}^L \geq w \mathit{k}^U \). Then,
\begin{align*}
\langle \nu, \tilde{\mathcal{K}} \rangle_{gH} + \langle w, \tilde{\mathcal{K}} \rangle_{gH} 
&= \langle \nu + w, \tilde{\mathcal{K}} \rangle_{gH}
\end{align*}
\end{itemize}

\vspace{0.2cm}
\noindent
Therefore, the \textbf{linearity of the gH-product} in its first component holds under above conditions.\\

\item[\textit{(iv)}]  Suppose \( \nu \neq 0 \) and \( \langle \nu, \tilde{\mathcal{K}} \rangle_{gH} = 0 \).  \vspace{0.2cm}\\
From the definition of the gH-product, we have
$$ \langle \nu, \tilde{\mathcal{K}} \rangle_{gH}=\left[ \min\{\nu \mathit{k}^L, \nu \mathit{k}^U\}, \max\{\nu \mathit{k}^L, \nu \mathit{k}^U\} \right] = [0, 0]$$
This holds if and only if $ \min\{\nu \mathit{k}^L, \nu \mathit{k}^U\} = 0$ ~~ \text{and} ~ $\max\{\nu \mathit{k}^L, \nu \mathit{k}^U\} = 0 $ which in turns hold  if and only if
$\nu \mathit{k}^L = 0$  \text{and} $ \nu \mathit{k}^U = 0.$\vspace{0.2cm}\\
We conclude that if  $\nu \neq 0$ , then  $\langle \nu, \tilde{\mathcal{K}} \rangle_{gH} = 0  \Leftrightarrow \nu \mathit{k}^L = 0 $ \text{and}  $\nu \mathit{k}^U = 0.$\vspace{0.2cm}\\
That is, $\nu \perp \mathit{k}^L$ ~ \text{and} ~ $\nu \perp \mathit{k}^U.$
\end{itemize}
\end{proof}

\vspace{0.2cm}
\noindent
The following example illutrates the linearity of the $gH$-product under the assumptions of \text{Proposition \ref{Prop 3.2.1}\textit{(iii)}}.
\begin{example}
    Let  
$
\nu = (1, -1), \omega = (-5, 4)
$
and  
$\tilde{\mathcal{K}} = (\mathcal{K}_1, \mathcal{K}_2), ~~ \text{where} ~~\mathcal{K}_1 = [1, 2],  \mathcal{K}_2 = [3, 4]$
Then,  
$\nu + \omega = (-4, 3)$\\
Now, compute the gH-products:
\begin{align}
\langle \nu + \omega, \tilde{\mathcal{K}} \rangle_{gH} &= \left[\min\{5, 4\}, \max\{5, 4\}\right] = [4, 5] \\
\langle \nu, \tilde{\mathcal{K}} \rangle_{gH} &= \left[\min\{-2, -2\}, \max\{-2, -2\}\right] = [-2, -2] \\
\langle \omega, \tilde{\mathcal{K}} \rangle_{gH} &= \left[\min\{7, 6\}, \max\{7, 6\}\right] = [6, 7]
\end{align}

Adding (3.2) and (3.3),we get
\[
\langle \nu, \tilde{\mathcal{K}} \rangle_{gH} + \langle \omega, \tilde{\mathcal{K}} \rangle_{gH} = [-2, -2] + [6, 7] = [4, 5]
\]

Hence,
$
\langle \nu + \omega, \tilde{\mathcal{K}} \rangle_{gH} = \langle \nu, \tilde{\mathcal{K}} \rangle_{gH} + \langle \omega, \tilde{\mathcal{K}} \rangle_{gH}
$
\end{example}
The following example illutrates the \text{Proposition \ref{Prop 3.2.1}\textit{(iv)}}.
\begin{example}
     Let  $v = \left(1, -2\right)$ and  $ \tilde{\mathcal{K}} = (\mathcal{K}_1, \mathcal{K}_2), \quad \text{where} \quad \mathcal{K}_1 = \left[3, 5\right],\\
     \quad  \mathcal{K}_2 = \left[\frac{3}{2}, \frac{5}{2}\right]$. So,\vspace{0.2cm}\\
  $v\mathit{k}^L=1\cdot3+(-2)\cdot\frac{3}{2}=0$  and  $v\mathit{k}^U=1\cdot5+(-2)\cdot\frac{5}{2}=0$\vspace{0.2cm}\\
This implies,  $
 \quad \langle v, \tilde{\mathcal{K}} \rangle_{gH} = 0 $  even if $ v \neq 0.
$\vspace{0.2cm}\\
Or, using definition of $gH$-product, we have
\begin{equation*}
    \begin{aligned}
        \langle v, \tilde{\mathcal{K}} \rangle_{gH}&=\mathcal{K}_1 \ominus_{gH}2\mathcal{K}_2\\
        &=\left[3, 5\right]\ominus_{gH}2\left[\frac{3}{2}, \frac{5}{2}\right]\\
        &=[0,0]
    \end{aligned}
\end{equation*}
\end{example}

\vspace{.2cm}   
\section{\textbf{CONCLUSION}}

In this paper, we have developed and analyzed a framework for computing the \emph{gH-gradient} of IVFs using the concept of \emph{gH differentiability}. Building upon foundational ideas from interval analysis, we first revisited the limitations of classical interval subtraction and highlighted the necessity of the \emph{gH-difference}. This framework allowed us to define consistent and well-behaved notions of \emph{gH-partial derivatives}, which act as the fundamental components of the \emph{gH-gradient} of IVFs.

We introduced the \emph{gH-product} of a vector with n-tuples of intervals. This operation plays a crucial role in analyzing optimality in uncertain or imprecise environments. Through multiple non-trivial examples, we demonstrated the correctness and applicability of the proposed  results.

The theoretical contributions of this work are significant in extending classical differential tools to interval-valued frameworks, thereby facilitating future developments in interval-valued optimization, uncertainty modeling, and interval-based variational analysis. Furthermore, the proposed \emph{gH-product} lays the groundwork for future extensions in higher-dimensional vector spaces and for applications in manifold-based optimization involving interval data.

Future research may focus on:
\begin{itemize}
    \item Extending this framework to higher-order derivatives and Hessians for IVFs.
    \item Applying the \emph{gH-gradient} framework in optimization algorithms involving interval-valued objective functions or constraints.
    \item Exploring the interaction between \emph{gH-differentiability} and \emph{generalized convexity} concepts such as E-invexity and geodesic convexity.
    \item Implementing numerical methods and algorithms to compute these derivatives for complex real-world problems involving uncertainty.
\end{itemize}

Since a \emph{fuzzy number} is canonically represented by a family of intervals (its $\alpha$-cuts), this characterization extends naturally and rigorously to \emph{fuzzy-valued functions}.

\vspace{.2cm}
\noindent
\textbf{Competing Interests:}
It is declare that all the authors have no conflict of interest and did not receive support
from any organization for the submitted work.

\vspace{.2cm}
\noindent
\textbf{Data Availability Statement:}
The study is purely theoretical. All illustrative examples are original and fully described in the manuscript; hence, no external data were used.


\begin{thebibliography}{25}

 \bibitem{AlefeldHerzberger} Alefeld, G.,  Herzberger, J. (1983). Introduction to Interval Computations. \textit{Academic Press, New York}.

\bibitem{Alikhani2014}  Alikhani,R., Bahrami, F., Gnana Bhaskar, T., (2014). Existence and uniqueness results for fuzzy linear differential-algebraic equations, \textit{Fuzzy Sets Syst.}245 (2014) 30–42.\\
https://doi.org/10.1016/j.fss.2014.03.006

\bibitem{Alikhani2017} Alikhani,R.,(2017) Interval fractional integrodifferential equations without singular kernel by fixed point in partially ordered sets, \textit{Comput. Methods}
Differ. Equ. 5 (2017) 12–29.

\bibitem{Allahviranloo2006} Allahviranloo, T., (2006). Difference methods for fuzzy partial differential equations, \textit{Comput. Methods Appl. Math.} 2 (2006) 233–242.\\
https://doi.org/10.2478/cmam-2002-0014
\bibitem{Allahviranloo2015}Allahviranloo,T., Gouyandeh, Z., Armand,A., Hasanoglu,A.,(2015). On fuzzy solutions for heat equation based on generalized Hukuhara differentiability, \textit{Fuzzy Sets Syst.} 265 (2015) 1–23.\\
https://doi.org/10.1016/J.FSS.2014.11.009

\bibitem{Alikhani2019} Alikhan,R., Bahram, F., Parvizi,S.(2019).
Differential calculus of fuzzy multi-variable functions and its applications to fuzzy partial differential equations,\textit{Fuzzy Sets and Systems},Volume 375,2019,Pages 100-120,ISSN 0165-0114,\\
https://doi.org/10.1016/j.fss.2019.04.011.

\bibitem{StefaniniBede2013} Bede,B., Stefanini,L.,(2013).Generalized differentiability of fuzzy-valued functions,\textit{Fuzzy Sets and Systems,}Volume 230,Pages 119-141,ISSN 0165-0114,\\
https://doi.org/10.1016/j.fss.2012.10.003.

\bibitem{Bhat2024} Bhat, H. A., Iqbal, A., Aftab, M. (2024). First and second order necessary conditions for multiobjective programming with interval-valued objective functions on Riemannian manifolds. \textit{RAIRO – Operations Research}, 58(5), 4259–4276.\\
https://doi.org/10.1051/ro/2024157

\bibitem{Bhat2025a} Bhat, H. A., Iqbal, A. (2025). Generalized Hukuhara directional differentiability of interval-valued functions on Riemannian manifolds. \textit{Optimization}, 1–27.\\
https://doi.org/10.1080/02331934.2024.2447996

\bibitem{Bhat2025b} Bhat, H. A., Iqbal, A., Aftab, M. (2025). Optimality conditions for interval-valued optimization problems on Riemannian manifolds under a total order relation. \textit{Journal of Optimization Theory and Applications}, 205(6).\\
https://doi.org/10.1007/s10957-025-02618-3

\bibitem{Chalco2011} Chalco-Cano, Y., Román-Flores, H.; Gamero, J. (2011). Generalized derivative and pi-derivative for set-valued functions. \textit{Inf. Sci}. 181, 2177–2188,\\ https://doi.org/10.1016/j.ins.2011.01.023.

\bibitem{Chalco2016} Chalco-Cano,Y., Rodríguez-López,R. and Jiménez-Gamero,M.D. (2016).
Characterizations of generalized differentiable fuzzy functions,
\textit{Fuzzy Sets and Systems,}Volume 295, Pages 37-56, ISSN 0165-0114,\\
https://doi.org/10.1016/j.fss.2015.09.005.
     	
 \bibitem{Doolittle}  Doolittle, E.K., Kerivin, H.L.M. \& Wiecek, M.M. (2018).Robust multiobjective optimization with application to Internet routing. \textit{Ann Oper Res} 271, 487–525.\\
 https://doi.org/10.1007/s10479-017-2751-5


\bibitem{Ghosh}Ghosh, D. (2017).  Newton method to obtain efficient solutions of the optimization problems with interval-valued objective functions. \textit{J. Appl. Math. Comput.} 53, 709–731 .\\ https://doi.org/10.1007/s12190-016-0990-2

\bibitem{Ghosh2019}Ghosh, D.,Singh,A.,Shukla,K.K., Manchanda,K.(2019).
Extended Karush-Kuhn-Tucker condition for constrained interval optimization problems and its application in support vector machines,
\textit{Information Sciences,}Volume 504,Pages 276-292,ISSN 0020-0255,\\
https://doi.org/10.1016/j.ins.2019.07.017.

\bibitem{Ghosh2020} Ghosh,D., Chauhan,R.S., Mesiar,R., Debnath,A.K.(2020).
Generalized Hukuhara Gâteaux and Fréchet derivatives of interval-valued functions and their application in optimization with interval-valued functions,
\textit{Information Sciences,}Volume 510,Pages 317-340,ISSN 0020-0255,\\
https://doi.org/10.1016/j.ins.2019.09.023.
        
\bibitem{Hukuhara}Hukuhara,M.(1967) Integration Des Applications Mesurables Dont La Valeur Est Un Compact Convexe, (in French) \textit{Funkcial. Ekvac.}, No. 10, pp. 205–223, 1967.


\bibitem{Izhar} Jaichander,R.R., Ahmad,I., Kummari,K. and Al-Homidan,S. (2022).Robust Nonsmooth Interval-Valued Optimization Problems Involving Uncertainty Constraints.

\bibitem{Markov}Markov, S.(1979). Calculus for interval functions of a real variable. \textit{Computing 22, 325–337 }\\
https://doi.org/10.1007/BF02265313

\bibitem{Moore1966} Moore, R. E. (1966). Interval Analysis. \textit{Prentice-Hall Englewood Cliff,NJ.}

\bibitem{Moore1979} Moore, R. E. (1979).Methods and Applications of Interval Analysis. \textit{SIAM, Philadelphia.}\\
https://doi.org/10.1137/1.9781611970906

 \bibitem{Gomez2022} Osuna-Gómez, R., Costa, T. M., Chalco-Cano, Y., Hernández-Jiménez, B.(2022). Quasilinear approximation for interval-valued functions via generalized Hukuhara differentiability. \textit{Computational and Applied Mathematics}. Volume 41 Number 4 Pages 149 Isbn 1807-0302\\
https://doi.org/10.1007/s40314-021-01746-6
\bibitem{Costa2022} Osuna-Gómez, R., Costa, T. M., Hernández-Jiménez, B.and Ruiz-Garzón,G.(2022). Necessary and sufficient conditions for interval-valued differentiability.\textit{Mathematical Methods in the Applied Sciences}.\\
https://doi.org/10.1002/mma.8647

 \bibitem{D. Qiu} Qiu,D. (2021). The generalized Hukuhara differentiability of interval-valued function is not fully equivalent to the one-sided differentiability of its endpoint functions, \textit{Fuzzy Sets and Systems},Volume 419,Pages 158-168,ISSN 0165-0114,\\
https://doi.org/10.1016/j.fss.2020.07.012.

\bibitem{Stefanini2009} Stefanini, L., Bede,B.(2009). Generalized Hukuhara differentiability of interval-valued functions and interval differential equations,\textit{Nonlinear Analysis: Theory, Methods \& Applications,}\\
Volume 71, Issues 3–4,2009,Pages 1311-1328,ISSN 0362-546X,\\
https://doi.org/10.1016/j.na.2008.12.005.


\bibitem{Stefanini2019} Stefanini,L., Arana-Jiménez,M.(2019). Karush–Kuhn–Tucker conditions for interval and fuzzy optimization in several variables under total and directional generalized differentiability,
\textit{Fuzzy Sets and Systems,}Volume 362,Pages 1-34,ISSN 0165-0114,\\
https://doi.org/10.1016/j.fss.2018.04.009.

\bibitem{Costa2023} Villanueva,F.,R.,  de Oliveira,V.,A.,  Costa,T.,M. (2023).
Optimality conditions for interval valued optimization problems,\textit{Fuzzy Sets and Systems,}Volume 454,Pages 38-55,ISSN 0165-0114,\\
https://doi.org/10.1016/j.fss.2022.06.020.

 \bibitem{Wu2007} Wu, H. C. (2007). The Karush-Kuhn-Tucker optimality conditions in an optimization problem with interval-valued objective function. \textit{European Journal of Operational Research}, 176, 46–59.\\
https://doi.org/10.1016/j.ejor.2005.09.007


       \end{thebibliography}
      \end{document}